\documentclass[11pt, leqno]{amsart}
\usepackage{amsmath,amsthm,amssymb}
\usepackage{comment}
\usepackage{tabularx}
\usepackage{color}
\usepackage{ulem}

\newcommand{\Hom}{\mathrm{Hom}}
\newcommand{\wt}[1]{\mathrm{wt}(#1)}
\newcommand{\supp}[1]{\mathrm{supp}(#1)}
\newcommand{\Ker}[1]{\mathrm{Ker}\, #1}

\newcommand{\rank}[1]{\mathrm{rank}\, #1}
\renewcommand{\Im}[1]{\mathrm{Im}\, #1}

\newcommand{\greektwo}{I\hspace{-1.2pt}I}
\newcommand{\greekthree}{I\hspace{-1.2pt}I\hspace{-1.2pt}I}
\newcommand{\greekfour}{I\hspace{-1.2pt}V}

\newtheorem{thm}{Theorem}[section]
\newtheorem{cor}[thm]{Corollary}
\newtheorem{lem}[thm]{Lemma}
\newtheorem{prop}[thm]{Proposition}
\newtheorem{conj}[thm]{Conjecture}
\theoremstyle{definition}
\newtheorem*{rem}{Remark}

\numberwithin{equation}{section}

\title[Jacobi polynomials for the first-order GRM codes]
{Jacobi polynomials for the first-order generalized Reed--Muller codes}

\author{Ryosuke Yamaguchi*}
\address{School of Fundamental Science and Engineering,
Waseda University,
Tokyo 169--8555, Japan
}
\email{ryosuke.yama.821@fuji.waseda.jp}



\begin{document}
\begin{abstract}
  In this paper, we give the Jacobi polynomials for 
  first-order generalized Reed--Muller codes.
We show as a corollary the nonexistence of combinatorial $3$-designs in these codes.
\end{abstract}
\maketitle

\section{Introduction}
There is growing interest in the designs derived from codes within the fields of coding theory and design theory.
In \cite{Ozeki,BMP}, a criterion was provided for determining whether a shell of a code constitutes a $t$-design, using 
Jacobi polynomials (see Proposition \ref{prop:design-Jacobi}, the definition of the Jacobi polynomials will be given in Section \ref{section:def-of-Jacobi}). 
Using this criterion, in \cite{BMP,CMOT}, they presented 
$t$-designs derived from Type \greektwo, \greekthree, and \greekfour codes with short lengths. Additionaly, 
Miezaki and Munemasa \cite{MM} provided Jacobi polynomials for the first-order Reed--Muller codes. As a corollary, they 
showed the nonexistence of combinatorial $4$-designs in these codes. 
The purpose of the present paper is to give a generalization of Miezaki and Munemasa's results.

Let $m$ be a positive integer and $q$ be a prime power, and set $V = \mathbb{F}_q^m$. 
The first-order generalized Reed--Muller (GRM) code $RM_q(1, m)$ is defined as the subspace of $\mathbb{F}_q^V$
consisting of affine linear functions:
\begin{equation*}
  RM_q(1, m) = \left\{(\lambda(x) + b)_{x \in V} \in \mathbb{F}_q^V \mid \lambda \in V^*, b \in \mathbb{F}_q\right\},
\end{equation*}
where $V^* = \Hom(V, \mathbb{F}_q)$.
We remark that the weight enumerator of $RM_q(1, m)$ is 
\begin{equation}\label{weight}
  x^{q^m} + (q^{m + 1} - q)x^{q^{m - 1}}y^{(q - 1)q^{m - 1}} + (q - 1)y^{q^m}.
\end{equation}

Let $C = RM_q(1, m)$ and $C_\ell := \{ c \in C \mid \wt{c} = \ell \}$. 
In this paper, we call $C_\ell$ a shell of the code $C$ whenever it is non-empty.
We show shells of $C$ are combinatorial $2$-designs but are not combinatorial $3$-designs by using Jacobi polynomials.
More precisely, the set $\mathcal{B}(C_\ell) := \{ \supp{x} \mid x \in C_\ell\}$ forms the set of blocks of a combinatorial $2$-design but does not form a $3$-design.
Herein, we always assume that a combinatorial $t$-design allows the existence of repeated blocks, and we exclude the
trivial design $\mathcal{D} = (\Omega, \mathcal{B})$ where $\Omega = \{1,\dots,n\}$ and $\mathcal{B} = \{\Omega,\dots,\Omega\}$.

\begin{rem}
  In \cite{MN}, Miezaki and Nakasora provided the first non-trivial examples of a code whose shells are $t$-designs for all weights and 
  whose shells are $t'$-designs for some weights with some $t' > t$. 
  Therefore, it is important to determine the value $t$ such that 
  all shells of a code are $t$-designs, and no shell is a $t+1$-design,
  if such a $t$ exists. 
\end{rem}

First, we provide Jacobi polynomials for $C$ with $T$, where $|T| = 2$.

\begin{thm}\label{thm:2-design}
  Let $C=RM_q(1, m)$ and $T = \{0, u\} \in \binom{V}{2}$. Then, 
  \begin{align*}
    J_{C, T}(w, z, x, y) =&
      w^2x^{q^m - 2} + (q^{m - 1} - 1)w^2x^{q^{m-1} - 2}y^{(q-1)q^{m-1}} \\
      &+ 2(q-1)q^{m-1}wzx^{q^{m-1}-1}y^{(q-1)q^{m-1}-1}\\
      &+ (q-1)(q^m-q^{m-1}-1)z^2x^{q^{m-1}}y^{(q-1)q^{m-1}-2} \\
      &+ (q-1)z^2y^{q^m - 2}.
  \end{align*}
\end{thm}

Using this theorem, we show that the shells of $RM_q(1, m)$ and $RM_q(1, m)^\perp$ are $2$-designs.

\begin{cor}\label{cor:2-design}
  Let $C=RM_q(1,m)$. Then for any $\ell \in \mathbb{N}$,
  $C_\ell$ is a combinatorial $2$-design.
  Similarly, $(C^\perp)_\ell$ is a combinatorial $2$-design.
\end{cor}

Second, we provide Jacobi polnomials for $C$ with $T$, where $|T| = 3$.

\begin{thm}\label{thm:3-design}
  Let $C=RM_q(1, m)$, $T=\{0, u_1, u_2\} \in \binom{V}{3}$, and $A = {}^t[u_1\ u_2]$.
  \begin{enumerate}
    \item If $\rank{A} = 2$, then
    \begin{align*}
      J_{C, T}(w, z, x, y) =& 
        w^3x^{q^m - 3} 
        + (q^{m - 2} - 1) w^3x^{q^{m - 1} - 3}y^{(q - 1)q^{m - 1}} \\
        &+ 3q^{m - 2}(q - 1) w^2zx^{q^{m - 1} - 2}y^{(q - 1)q^{m - 1} - 1} \\
        &+ 3q^{m - 2}(q - 1)^2 wz^2x^{q^{m - 1} - 1}y^{(q - 1)q^{m - 1} - 2} \\
        &+ (q - 1)(q^m - 2q^{m - 1} + q^{m - 2} - 1) z^3x^{q^{m - 1}}y^{(q - 1)q^{m - 1} - 3} \\
        &+ (q - 1)z^3y^{q^m - 3}.
    \end{align*}
    \item If $\rank{A} = 1$, then
    \begin{align*}
      J_{C, T}(w, z, x, y) =& 
        w^3x^{q^m - 3} 
        + (q^{m - 1} - 1) w^3x^{q^{m - 1} - 3}y^{(q - 1)q^{m - 1}} \\
        &+ 3q^{m - 1}(q - 1) wz^2x^{q^{m - 1} - 1}y^{(q - 1)q^{m - 1} - 2} \\
        &+ (q - 1)(q^m - 2q^{m - 1} - 1) z^3x^{q^{m - 1}}y^{(q - 1)q^{m - 1} - 3} \\
        &+ (q - 1)z^3y^{q^m - 3}.
    \end{align*}
  \end{enumerate}
\end{thm}

We show, as a corollary, the nonexistence of combinatorial $3$-designs in these codes.

\begin{cor}\label{cor:3-design}
  Let $C=RM_q(1,m)$. If $q \geq 3$ and $m \geq 2$, then for any $\ell \in \mathbb{N}$,
  $C_\ell$ is not a combinatorial $3$-design.
\end{cor}

Using Theorem~\ref{thm:3-design}, we show that $C_\ell$ is a 
$3$-$(v, k, (\lambda_1, \lambda_2))$-design (see Section \ref{section:def-of-design}).

\begin{cor}\label{cor:3-gen-design}
  Let $C = RM_q(1, m)$ and $\ell = (q - 1)q^{m-1}$. Then, $C_\ell$ is a combinatorial
  $3$-$(v, k, (\lambda_1, \lambda_2))$-design, where
  \begin{align*}
    v &= q^m,\ k = \ell = (q - 1)q^{m-1},\\
    \lambda_1 &= (q - 1)(q^m - 2q^{m - 1} + q^{m - 2} - 1),\\
    \lambda_2 &= (q - 1)(q^m - 2q^{m - 1} - 1).
  \end{align*}
\end{cor}

Third, we provide Jacobi polnomials for $C$ with $T$, where $|T| = 4$.

\begin{thm}\label{thm:4-design}
  Let $C=RM_q(1, m)$, $T=\{0, u_1, u_2, u_3\} \in \binom{V}{4}$, and $A = {}^t[u_1\ u_2\ u_3]$.
  \begin{enumerate}
    \item If $\rank{A} = 3$ then, 
          \begin{align*}
            J_{C, T}(w, z, x, y) = &w^4 x^{q^m - 4} + (q^{m-3} - 1)w^4x^{q^{m - 1} - 4}y^{(q - 1)q^{m - 1}} \\
            &+ 4q^{m - 3}(q - 1)w^3zx^{q^{m - 1} - 3}y^{(q - 1)q^{m - 1} - 1} \\
            &+ 6(q - 1)^2q^{m - 3}w^2z^2x^{q^{m - 1} - 2}y^{(q - 1)q^{m - 1} - 2} \\
            &+ 4q^{m - 3}(q - 1)^3wz^3x^{q^{m - 1} - 1}y^{(q - 1)q^{m - 1} - 3} \\
            &+ (q - 1)(q^m - 3q^{m - 1} + 3q^{m - 2} - q^{m - 3} - 1)z^4x^{q^{m - 1}}y^{(q - 1)q^{m - 1} - 4} \\
            &+ (q - 1)z^4y^{q^m - 4}.
          \end{align*}
    \item If $\rank{A} = 2$ then, 
          \begin{align*}
            J_{C, T}(w, z, x, y) = &w^4 x^{q^m - 4} + (q^{m - 2} - 1)w^4x^{q^{m - 1} - 4}y^{(q - 1)q^{m - 1}} \\
            &+ q^{m - 2}(q - 1) w^3zx^{q^{m - 1} - 3}y^{(q - 1)q^{m - 1} - 1} \\
            &+ 3q^{m - 2} (q - 1) w^2z^2x^{q^{m - 1} - 2}y^{(q - 1)q^{m - 1} - 2} \\
            &+ q^{m - 2}(q - 1)(4q - 5) wz^3x^{q^{m - 1} - 1}y^{(q - 1)q^{m - 1} - 3} \\
            &+ (q - 1)(q^m - 3q^{m - 1} + 2q^{m - 2} - 1) z^4x^{q^{m - 1}}y^{(q - 1)q^{m - 1} - 4} \\
            &+ (q - 1)z^4y^{q^m - 4}
          \end{align*}
          or
          \begin{align*}
            J_{C, T}(w, z, x, y) = &w^4 x^{q^m - 4} + (q^{m - 2} - 1)w^4x^{q^{m - 1} - 4}y^{(q - 1)q^{m - 1}} \\
            &+ 6q^{m - 2}(q - 1)w^2z^2x^{q^{m - 1} - 2}y^{(q - 1)q^{m - 1} - 2} \\
            &+ q^{m - 2}(q - 1)(4q - 8) wz^3x^{q^{m - 1} - 1}y^{(q - 1)q^{m - 1} - 3} \\
            &+ (q - 1)(q^m - 3q^{m - 1} + 3q^{m - 2} - 1) z^4x^{q^{m - 1}}y^{(q - 1)q^{m - 1} - 4} \\
            &+ (q - 1)z^4y^{q^m - 4}.
          \end{align*}
    \item If $\rank{A} = 1$ then, 
          \begin{align*}
            J_{C, T}(w, z, x, y) = &w^4 x^{q^m - 4} + (q^{m - 1} - 1)w^4x^{q^{m - 1} - 4}y^{(q - 1)q^{m - 1}} \\
            &+ 4q^{m - 1}(q - 1)wz^3x^{q^{m - 1} - 1}y^{(q - 1)q^{m - 1} - 3} \\
            &+ (q - 1)(q^m - 3q^{m - 1} - 1) z^4x^{q^{m - 1}}y^{(q - 1)q^{m - 1} - 4} \\
            &+ (q - 1)z^4y^{q^m - 4}.
          \end{align*}
  \end{enumerate}
\end{thm}

By this theorem, we show that $C_\ell$ is a combinatorial 
$4$-$(v,k,(\lambda_1,\lambda_2,\lambda_3,\lambda_4))$-design.

\begin{cor}\label{cor:4-gen-design}
  Let $C = RM_q(1, m)$ and $\ell = (q-1)q^{m-1}$. Then $C_\ell$ is a $4$-$(v, k, (\lambda_1,\lambda_2,\lambda_3,\lambda_4))$-design, where
  \begin{align*}
    v &= q^m, \ k = \ell = (q-1)q^{m-1},\\
    \lambda_1 &=  (q - 1)(q^m - 3q^{m - 1} + 3q^{m - 2} - q^{m - 3} - 1), \\
    \lambda_2 &= (q - 1)(q^m - 3q^{m - 1} + 2q^{m - 2} - 1), \\
    \lambda_3 &= (q - 1)(q^m - 3q^{m - 1} + 3q^{m - 2} - 1),\\
    \lambda_4 &= (q - 1)(q^m - 3q^{m - 1} - 1).
  \end{align*}
\end{cor}

This paper is organized as follows. 
In Section~\ref{sec:pre}, 
we define and give some basic properties of codes, 
combinatorial $t$-designs, and
Jacobi polynomials
used in this paper.
In Sections~\ref{sec:2-design}, \ref{sec:3-design}, and \ref{sec:4-design}, 
we show Theorems~\ref{thm:2-design}, \ref{thm:3-design}, and \ref{thm:4-design}, 
respectively.

All computer calculations reported in this paper were carried out using 
{\sc Magma}~\cite{Magma} and {\sc Mathematica}~\cite{Mathematica}. 

\section{Preliminaries}\label{sec:pre}
In this section, we give definitions and some properties of codes, combinatorial designs, and Jacobi polynomials.
We mainly refer to \cite{MM} and \cite{CMOT}.

\subsection{Codes and combinatorial $t$-designs.}\label{section:def-of-design}
Let $q$ be a prime power. A  $q$-ary linear code $C$ of length $n$ is a linear subspace of $\mathbb{F}_q^n$. 
The dual code of $C$ is the set of vectors which are orthogonal to any codewords in $C$: 
$C^{\perp} := \{x \in \mathbb{F}_q^n \mid x \cdot c = 0\ \mathrm{for\ all\ } c \in C\}$.
For $c \in \mathbb{F}_q^n$, the weight $\wt{c}$ is the number of its nonzero components.
The shell of a weight-$\ell$ is the set of codewords whose weight is $\ell$: $C_\ell :=
\{ c \in C \mid \wt{c} = \ell \}$.

A combinatorial $t$-design is a pair $\mathcal{D} = (\Omega, \mathcal{B})$, where $\Omega$ is a set of points of 
cardinality $v$, and $\mathcal{B}$ is a collection of $k$-element subsets of $\Omega$ called blocks, 
with the property that any $t$-element subset of $\Omega$ is contained in precisely $\lambda$ blocks.
Recall \cite{BMP} for the definitions of various types of designs. 
A pair $\mathcal{D} = (\Omega, \mathcal{B})$ is a combinatorial design with parameters $t$-$(v, k, (\lambda_1,\dots,\lambda_N))$ 
if $\mathcal{B}$ is a collection of $k$-element subsets of $\Omega$ called blocks and every $t$-element subset of $\Omega$ is contained 
in $\lambda_i$ blocks. Note that for $N=1$, the design coincides exactly with a $t$-design.

The support of a vector $x := (x_1,\dots,x_n)$, $x_i \in \mathbb{F}_q$, is the set of 
indices of its nonzero coordinates: $\supp{x} = \{ i \mid x_i \neq 0\}$. Let $\Omega := \{1,\dots,n\}$ and 
$\mathcal{B}(C_\ell):= \{\supp{x}\mid x \in C_\ell\}$. Then for a code $C$ of length $n$, we say that the shell $C_\ell$
is a combinatorial $t$-design if $(\Omega, \mathcal{B}(C_\ell))$ is a combinatorial $t$-design.
Similarly, we say that the shell $C_\ell$ is a combinatorial $t$-$(v, k, (\lambda_1,\dots,\lambda_N))$-design  
if $(\Omega, \mathcal{B}(C_\ell))$ is a combinatorial $t$-$(v, k, (\lambda_1,\dots, \lambda_N))$-design.

\subsection{Jacobi poynomials.}\label{section:def-of-Jacobi}
Let $C$ be a $q$-ary code of length $n$ and $T \subset [n] := \{1,\dots,n\}$. Then the
Jacobi polynomial of $C$ with $T$ is defined as follows \cite{Ozeki}: 
\begin{equation*}
  J_{C, T}(w, z, x, y) := \sum_{c \in C} w^{m_0(c)}z^{m_1(c)}x^{n_0(c)}y^{n_1(c)},
\end{equation*}
where for $c = (c_1,\dots,c_n)$,
\begin{align*}
  m_0(c) &= |\{ j \in T \mid c_j = 0\}|,\\
  m_1(c) &= |\{ j \in T \mid c_j \neq 0 \}|,\\
  n_0(c) &= |\{ j \in [n] \setminus T \mid c_j = 0\}|,\\
  n_1(c) &= |\{ j \in [n] \setminus T \mid c_j \neq 0\}|.
\end{align*}
The Jacobi polynomial of the dual code is written as follows. 
\begin{thm}[Theorem 4 of \cite{Ozeki}]\label{thm:dual-code}
  Let $C$ be a $q$-ary code of length $n$ and $T \subset [n]$. Then we have
  \begin{equation*}
    J_{C^\perp, T}(w, z, x, y) = \frac{1}{|C|}J_{C, T}(w + (q - 1)z, w - z, x + (q - 1)y, x - y).
  \end{equation*}
\end{thm}

Clearly, we have the following relation between Jacobi polynomials and combinatorial designs.
\begin{prop}\label{prop:design-Jacobi}
  Let $C$ be a linear code. $C_\ell$ is a combinatorial $t$-design if and only if 
  the coefficient of $z^tx^{n-\ell}y^{\ell-t}$ in $J_{C, T}$ is independent of the choice of $T$ with 
  $|T| = t$. 
\end{prop}
We remark that the Jacobi polynomials are invariant under the automorphisms of codes.
In particular, for $C = RM_q(1, m)$, a map $\varphi$ such that $\varphi(\lambda(x) + b') = \lambda(x) + b + b'$ is 
an automorphism on $C$. Then, we have the following result.
\begin{prop}\label{prop:Jacobi-inv}
  Let $C = RM_q(1, m)$, $T \subset V$, and $T' = T + v$, where $v \in V$.
  Then we have $J_{C, T} = J_{C, T'}$.
\end{prop}
This fact is useful because it suffices to consider the case that $T$ contains zero.

\subsection{Notation}\label{subsec:notations}
We next introduce some notation. Let $C = RM_q(1, m)$, $T \subset V$, and $t = |T|$.
For $c = (\lambda(x) + b)_{x \in V} \in C$ and $u \in V$, the evaluation of $c$ at $u$ is denoted by $c(u)$, 
which equals $\lambda(u) + b$. 
Let $i \in \{0,1,\dots,t\}$, $j \in \mathbb{F}_q$, we define
\begin{gather*}
  n_{j, T}(c) := |\{u \in T \mid c(u) = j\}|, \\
  b_{i, j} := | \{c \in V^* \mid n_{j, T}(c) = i\} |,
\end{gather*}
and
\begin{gather*}
  a_i := | \{c \in C \setminus \mathbb{F}_q\mathbf{1} \mid \mathrm{wt}(c|_T) = i\} |, \\
  b_i := \sum_{j \in \mathbb{F}_q} b_{i, j} = \sum_{j \in \mathbb{F}_q} | \{c \in V^* \mid n_{j, T}(c) = i\} |.
\end{gather*}
Then, the Jacobi polynomial of $C$ with $T$ is written as 
\begin{equation}\label{eq:Jacobi}
  \begin{split}
  J_{C, T}(w, z, x, y) &= w^tx^{q^m-t} \\
  &+ \sum_{i = 0}^t a_i w^{t-i}z^ix^{q^{m-1}-(t-i)}y^{(q-1)q^{m-1}-i} \\
  &+ (q-1)z^ty^{q^m-t}. 
  \end{split}
\end{equation}
We have the following relation between $a_i$ and $b_i$.
\begin{lem}\label{lem:b-to-a}
  We have
  \begin{equation*}
    a_i = b_{t - i} - \delta_{i, 0} - (q - 1)\delta_{i, t}.
  \end{equation*}
\end{lem}
\begin{proof}
  Since
  \begin{align*}
    a_{t-i} &= |\{c \in C \setminus \mathbb{F}_q \mathbf{1} \mid n_{0,  T}(c) = i\}| \\
          &= \sum_{j \in \mathbb{F}_q } |\{c \in (V^* + j\mathbf{1}) \setminus \{j\mathbf{1}\} \mid n_{0, T}(c) = i\}| \\
          &= \sum_{j \in \mathbb{F}_q } |\{c \in V^* \setminus\{0\} \mid n_{0, T}(c + j\mathbf{1}) = i \}| \\
          &= \sum_{j \in \mathbb{F}_q } |\{c \in V^* \setminus\{0\} \mid n_{-j, T}(c) = i \}| \\
          &= |\{ c \in V^* \setminus \{0\} \mid n_{0, T}(c) = i\} | + \sum_{j \in \mathbb{F}_q \setminus\{0\}} | \{ c \in V^* \setminus\{0\} \mid n_{-j, T}(c) = i \} | \\
          &= b_{i, 0} - \delta_{i, t} + \sum_{j \in \mathbb{F}_q \setminus \{0 \} } (b_{i, -j} - \delta_{i, 0}) \\
          &= \sum_{j \in \mathbb{F}_q} b_{i, -j} - \delta_{i, t} - (q - 1)\delta_{i, 0} \\
          &= \sum_{j \in \mathbb{F}_q} b_{i, j} - \delta_{i, t} - (q - 1)\delta_{i, 0} \\
          &= b_i - \delta_{i, t} - (q - 1)\delta_{i, 0},
  \end{align*}
  we obtain 
  \begin{align*}
    a_i &= b_{t - i} - \delta_{t - i, t} - (q - 1)\delta_{t - i, 0} \\
        &= b_{t - i} - \delta_{i, 0} - (q - 1)\delta_{i, t}. \\
  \end{align*}
\end{proof}

Using this lemma, we obtain $a_i$ by calculating $b_i$.

\section{Proofs of Theorem~\ref{thm:2-design} and Corollary~\ref{cor:2-design}}\label{sec:2-design}
In this section, we give proofs of Theorem~\ref{thm:2-design} and Corollary~\ref{cor:2-design} 
using the notation introduced in Section~\ref{subsec:notations}. First, we give a lemma to show Theorem~\ref{thm:2-design}.
Let $T = \{ 0, u \} \in \binom{V}{2}$.

\begin{lem}\label{lem:2-design}
  We have
  \begin{align*}
    b_0 &= q^{m - 1}(q - 1)^2, \\
    b_1 &= 2q^{m - 1}(q - 1), \\
    b_2 &= q^{m - 1}.
  \end{align*}
\end{lem}

\begin{proof}
Considering $u \in T$ as an element of $V^{**}$,   $u$ is 
a surjective linear map from $V^*$ to $\mathbb{F}_q$ because $u \neq 0$.
Then, for $j \in \mathbb{F}_q \setminus \{0\}$,  
\begin{gather*}
  b_{0, 0} = 0,\ b_{0, j} = \left| \bigcup_{a\in\mathbb{F}_q \setminus\{j\}}u^{-1}(a) \right| = (q-1)q^{m-1},\\
  b_{1, 0} = (q-1)q^{m-1},\ b_{1, j} = \left| u^{-1}(j) \right| = q^{m-1},\\
  b_{2, 0} = q^{m-1},\ b_{2, j} = 0.
\end{gather*}  
Therefore,
\begin{align*}
  b_0 &= b_{0,0} + \sum_{j\in\mathbb{F}_q\setminus\{0\}}b_{0,j} = (q-1)^2q^{m-1},\\
  b_1 &= b_{1,0} + \sum_{j\in\mathbb{F}_q\setminus\{0\}}b_{1,j} = 2(q-1)q^{m-1},\\
  b_2 &= b_{2,0} + \sum_{j\in\mathbb{F}_q\setminus\{0\}}b_{2,j} = q^{m-1}.
\end{align*}
\end{proof}

Using this Lemma, we show Theorem~\ref{thm:2-design}.

\begin{proof}[Proof of Theorem~\ref{thm:2-design}]
  Using Lemmas~\ref{lem:b-to-a} and \ref{lem:2-design},
  we obtain
  \begin{align*}
    a_0 &= b_2 - 1 = q^{m-1} - 1,\\
    a_1 &= b_1 = 2(q-1)q^{m-1},\\
    a_2 &= b_0 - (q-1) = (q-1)(q^m - q^{m-1} - 1).
  \end{align*}
  Thus, we obtain coefficients of the Jacobi polynomial by (\ref{eq:Jacobi}).
\end{proof}

Finally, we give a proof of Corollary~\ref{cor:2-design}.

\begin{proof}[Proof of Corollary~\ref{cor:2-design}]
  By Proposition~\ref{prop:Jacobi-inv}, it suffices to show that for any $T = \{ 0, u\} \subset \binom{V}{2}$, the coefficient of $z^2x^{q^m - \ell}y^{\ell-2}$ is the same value.
  This is true by using Theorem~\ref{thm:2-design}.
  In addition, we have the Jacobi polynomial of $RM_q(1, m)^\perp$ by using Theorem~\ref{thm:dual-code}. Thus, we obtain the desired results.
\end{proof}

\begin{rem}
Collorary \ref{cor:2-design} can be proved by $2$-transitivity of the automorphism group
of $RM_q(1, m)$, which is the general linear homogenious group \cite{Berger-Charpin}.
\end{rem}

\section{Proofs of Theorem~\ref{thm:3-design} and Corollary~\ref{cor:3-design}}\label{sec:3-design}
In this section, we give proofs of Theorem~\ref{thm:3-design} and Corollaries~\ref{cor:3-design} and \ref{cor:3-gen-design} 
using the notation introduced in Section~\ref{subsec:notations}.
First, we give two lemmas to show Theorem~\ref{thm:3-design}.
Let $T = \{0, u_1, u_2\}\in \binom{V}{3}$, $A = {}^t[u_1\ u_2]$.
\begin{lem}\label{lem:3-design-rank2}
If $\rank{A} = 2$, then
\begin{align*}
  b_0 &= q^{m - 2}(q - 1)^3, \\
  b_1 &= 3q^{m - 2}(q - 1)^2, \\
  b_2 &= 3q^{m - 2}(q - 1), \\
  b_3 &= q^{m - 2}.
\end{align*}
\end{lem}
\begin{proof}
We remark that $A$ is a surjective map from $V$ to $\mathbb{F}_q^2$, 
and for all $ a, b \in \mathbb{F}_q, |A^{-1}({}^t[a\ b])| = |\Ker{A}| = q^{m - 2}$.
Then, for $j \in \mathbb{F}_q \setminus \{0\}$,
\begin{align*}
  b_{0, j} &= \left|\bigcup_{a,b\neq j} A^{-1}({}^t[a\ b])\right|
          =\sum_{a, b\in \mathbb{F}_q \setminus \{j\}} \left| A^{-1}({}^t [a\ b]) \right| 
          = (q - 1)^2q^{m - 2} ,\\
  b_{1, j} &= 2\times \left|\bigcup_{a\neq j} A^{-1} ({}^t[j\ a])\right|
          = 2(q - 1)q^{m - 2},\\
  b_{2, j} &= \left| A^{-1} ({}^t[j\ j]) \right|
          = q^{m - 2},\\
  b_{3, j} &= 0
\end{align*}
and 
\begin{align*}
  b_{0,0} &= 0,\\
  b_{1,0} &= \left|\bigcup_{a,b\neq 0} A^{-1}({}^t[a\ b])\right|
          =\sum_{a, b\in \mathbb{F}_q \setminus \{0\}} \left| A^{-1}({}^t [a\ b]) \right| 
          = (q - 1)^2q^{m - 2},\\
  b_{2,0} &=  2\times \left|\bigcup_{a\neq 0} A^{-1} ({}^t[j\ a])\right|
          = 2(q - 1)q^{m - 2},\\
  b_{3,0} &= \left| \Ker{A} \right| = q^{m - 2}.
\end{align*}
Since
\begin{equation*}
  b_i = b_{i, 0} + \sum_{j \in \mathbb{F}_q\setminus \{0\}}b_{i, j},
\end{equation*}
we obtain the desired results.
\end{proof}

\begin{lem}\label{lem:3-design-rank1}
If $\rank{A} = 1$, then
\begin{align*}
  b_0 &= q^{m - 1}(q - 1)(q - 2), \\
  b_1 &= 3q^{m - 1}(q - 1), \\
  b_2 &= 0, \\
  b_3 &= q^{m - 1}.
\end{align*}
\end{lem}

\begin{proof}
Because $u_1$ and $u_2$ are not equal to $0$ or each other, 
there exist $a, b \in \mathbb{F}_q \setminus \{0\}$ such that $a \neq b$ and
 $\{{}^t[a\ b] \}$ is a basis of $\Im{A}$.
Hence, for any $v \in \Im{A}$,
\begin{equation*}
  |A^{-1}(v)| = |\Ker{A}| = q^{\dim{\Ker{A}}} = q^{m - \rank{A}} = q^{m - 1}.
\end{equation*}
Then, for $j \in \mathbb{F}_q \setminus \{0\}$, we have
\begin{align*}
  b_{1, j} &= |A^{-1}({}^t[j\ ja^{-1}b])| + |A^{-1}({}^t[jb^{-1}a\ j])| = 2q^{m - 1},\\
  b_{2, j} &= b_{3, j} = 0, \\
  b_{0, j} &= |V| - (b_{1, j} + b_{2, j} + b_{3, j}) = q^m - 2q^{m - 1}
\end{align*}
and 
\begin{align*}
  b_{1, 0} &= (q - 1)q^{m - 1},\\
  b_{0, 0} &= b_{2, 0} = 0, \\
  b_{3, 0} &= |\Ker{A}| = q^{m - 1}.
\end{align*}
Thus, since $b_i = \sum_{j \in \mathbb{F}_q} b_{i, j}$, we obtain the desired results.
\end{proof}

Using these lemmas, we give a proof of Theorem~\ref{thm:3-design}.

\begin{proof}[Proof of Theorem~\ref{thm:3-design}]
  By Lemmas~\ref{lem:b-to-a} and (\ref{eq:Jacobi}), (1) follows from Lemma~\ref{lem:3-design-rank2}, and (2) follows from Lemma~\ref{lem:3-design-rank1}.  
\end{proof}

Then, we show that the shells of $RM_q(1, m)$ are not $3$-designs if $q \geq 3$ and $m \geq 2$.

\begin{proof}[Proof of Corollary~\ref{cor:3-design}]
  Let $C = RM_q(1, m)$. We give a proof relying on the properties of Jacobi polynomials. Let 
  $T_1 = \{ 0, u_1, u_2 \} \in \binom{V}{3}$, $T_2 = \{ 0, v_1, v_2 \} \in \binom{V}{3}$, $A_1 = {}^t[u_1\ u_2]$,
  and $A_2 = {}^t[v_1 \ v_2]$.
  We assume that $\rank{A_1} = 2$ and $\rank{A_2} = 1$. Indeed, if $q \geq 3$ and $m \geq 2$, there exist such $T_1, T_2$.
  By Theorem~\ref{thm:3-design},
  \begin{equation*}
       J_{C,T_1} - J_{C,T_2} 
      = -q^{m-2}(q-1)x^{q^{m-1}-3}y^{(q-1)q^{m-1}-3}(wy - xz)^3.
  \end{equation*}
  Since the coefficient of $z^3x^{q^m-\ell}y^{\ell-3}$ in $J_{C, T_1} - J_{C, T_2}$ is non-zero whenever $C_\ell$ is
  non-empty, $C_\ell$ is not a $3$-design.
\end{proof}

 By using Theorem \ref{thm:dual-code}, we obtain
\begin{align*}
J_{C^\perp, T_1} - J_{C^\perp, T_2} 
= (q - 1)\{x + (q - 1)y\}^{q^{m-1}-3}(x - y)^{(q-1)q^{m-1}-3}(wy - xz)^3.
\end{align*}
Based on this equation, we conjecture the following.

\begin{conj}
  Let $C = RM_q(1, m)$. If $q \geq 3$, then for any $\ell \in \mathbb{N}$, 
  $(C^\perp)_\ell$ is not a combinatorial $3$-design.
\end{conj}

We verified this conjecture for $q$, $m$ which satisfying $q^{2m} < 10^9$.
The computations were performed using the code available at GitHub\footnote[1]{https://github.com/yama821/GRMJacobi-paper}.

Finally, we claim that the shells of 
$RM_q(1, m)$ are $3$-$(v,k,(\lambda_1, \lambda_2))$-designs.
\begin{proof}[Proof of Corollary~\ref{cor:3-gen-design}]
  It is clear from Theorem~\ref{thm:3-design}.
\end{proof}

\section{Proof of Theorem~\ref{thm:4-design}}\label{sec:4-design}
In this section, we give proofs of Theorem~\ref{thm:4-design} and Corollary~\ref{cor:4-gen-design} 
using the notation introduced in Section~\ref{subsec:notations}. Let $T = \{0, u_1, u_2, u_3\} \in \binom{V}{4}$, 
$A = {}^t[u_1\ u_2\ u_3]$.

Considering $A$ as a linear map from $V$ to $\mathbb{F}_q^3$, 
for all $j \in \mathbb{F}_q\setminus \{0\}$, we have
\begin{equation}\label{eq:b_ij}
  \begin{split}
  b_{i, j} &= |\{ c \in V \mid i \text{ elements of $Ac$ are equal to }j\}| \\
  &= q^{m - \rank{A}} \times |\{ v \in \Im{A} \mid i \text{ elements of $v$ are equal to } j\}|,
  \end{split}
\end{equation}
and for $j = 0$, we have
\begin{equation}\label{eq:b_i0}
  \begin{split}
  b_{i, 0} &= |\{ c \in V \mid i - 1 \text{ elements of $Ac$ are equal to } 0\}| \\
           &= q^{m - \rank{A}} \times |\{ v \in \Im{A} \mid i - 1 \text{ elements of $v$ are equal to } 0\}|. 
  \end{split}
\end{equation}
Next, we prepare three lemmas for proving Theorem~\ref{thm:4-design}.

\begin{lem}\label{lem:4-design-rank3}
  If $\rank{A} = 3$,
  \begin{gather*}
    a_0 = q^{m - 3} - 1,\ 
    a_1 = 4(q - 1)q^{m - 3} ,\ 
    a_2 = 6(q - 1)^2q^{m - 3},\\
    a_3 = 4(q - 1)^3q^{m - 3},\ 
    a_4 = (q - 1)(q^m - 3q^{m - 1} + 3q^{m - 2} - q^{m - 3} - 1).
  \end{gather*}
\end{lem}

\begin{proof}
  By (\ref{eq:b_ij}) and (\ref{eq:b_i0}), 
  \begin{align*}
    b_0 &= b_{0, 0} + \sum_{j \in \mathbb{F}_q \setminus \{0\}} b_{0, j} 
        = (q - 1) \times (q - 1)^3 \times q^{m - 3}
        = (q - 1)^4q^{m - 3}, \\
    b_1 &= b_{1, 0} + \sum_{j \in \mathbb{F}_q \setminus \{0\}} b_{1, j} 
         = (q - 1)^3q^{m - 3} + 3(q - 1)^3q^{m - 3}
         = 4(q - 1)^3q^{m - 3}, \\
    b_2 &= b_{2, 0} + \sum_{j \in \mathbb{F}_q \setminus \{0\}} b_{2, j} 
         = 3(q - 1)^2q^{m - 3} + 3(q - 1)^2q^{m - 3} 
         = 6(q - 1)^2q^{m - 3}, \\
    b_3 &= b_{3, 0} + \sum_{j \in \mathbb{F}_q \setminus \{0\}} b_{3, j} 
         = 3(q - 1)q^{m - 3} + (q - 1)q^{m - 3} 
         = 4(q - 1)q^{m - 3}, \\
    b_4 &= b_{4, 0} + \sum_{j \in \mathbb{F}_q \setminus \{0\}} b_{4, j} 
         = q^{m - 3}.
  \end{align*}
  Therefore, we obtain the desired results by Lemma~\ref{lem:b-to-a}.
\end{proof}

\begin{lem}\label{lem:4-design-rank1}
  If $\rank{A} = 1$, 
  \begin{gather*}
    a_0 = q^{m - 1} - 1,\ 
    a_1 = 0,\
    a_2 = 0,\\
    a_3 = 4(q - 1)q^{m - 1} ,\
    a_4 = (q - 1)(q^m - 3q^{m - 1} - 1).
  \end{gather*}
\end{lem}

\begin{proof}
  Since $u_1, u_2, u_3$ are different from each other and not equal to $0$, we take
  a basis $\{{}^t[a\ b\ c]\}$ of $\Im{A} $, where $a, b, c \in \mathbb{F}_q\setminus\{0\}$ 
  and are different from each other. If $\rank{A} = 1$, then for all $v \in \Im{A}$, $|A^{-1}(v)| = q^{m-1}$. 
  Then, for all $j \in \mathbb{F}_q\setminus\{0\}$, 
  \begin{align*}
    &b_{1, j} = 3 \times q^{m - 1} = 3q^{m - 1},\\
    &b_{2, j} = b_{3, j} = b_{4, j} = 0,\\
    &b_{0, j} = q^m - (b_{1, j} + b_{2, j} + b_{3, j} + b_{4, j}) = q^m - 3q^{m - 1}
  \end{align*}
  and 
  \begin{align*}
    &b_{1, 0} = (q - 1) \times q^{m - 1}, \\
    &b_{0, 0} = b_{2, 0} = b_{3, 0} = 0, \\
    &b_{4, 0} = q^{m - 1}.
  \end{align*}
  Therefore, $b_0,\dots,b_4$ are written as follows:
  \begin{align*}
    b_0 &= b_{0, 0} + \sum_{j \in \mathbb{F}_q \setminus \{0\}} b_{0, j} 
        = (q - 1) \times (q^m - 3q^{m - 1})
        = (q - 1)(q - 3)q^{m - 1}, \\
    b_1 &= b_{1, 0} + \sum_{j \in \mathbb{F}_q \setminus \{0\}} b_{1, j} 
         = (q - 1)q^{m - 1} + (q - 1) \times 3q^{m - 1} 
         = 4(q - 1)q^{m - 1}, \\
    b_2 &= b_{2, 0} + \sum_{j \in \mathbb{F}_q \setminus \{0\}} b_{2, j} 
         = 0, \\
    b_3 &= b_{3, 0} + \sum_{j \in \mathbb{F}_q \setminus \{0\}} b_{3, j} 
         = 0, \\
    b_4 &= b_{4, 0} + \sum_{j \in \mathbb{F}_q \setminus \{0\}} b_{4, j} 
         = q^{m - 1}.
  \end{align*}
  Thus, we obtain $a_0,\dots,a_4$ by using Lemma~\ref{lem:b-to-a}.
\end{proof}

Before stating the lemma under the condition of $\rank{A} = 2$, we give a basis of 
$\Im{A}$. Since $\Im{A}$ is invariant under the right multiplication of an invertible
matrix to $A$, we confine a basis of $\Im{A}$ to the following:
\begin{equation*}
  \left\{\begin{bmatrix}1\\0\\a\end{bmatrix}, \begin{bmatrix}0\\1\\b\end{bmatrix}\right\}, 
  \left\{\begin{bmatrix}1\\a\\0\end{bmatrix}, \begin{bmatrix}0\\0\\1\end{bmatrix}\right\}, 
  \left\{\begin{bmatrix}0\\1\\0\end{bmatrix}, \begin{bmatrix}0\\0\\1\end{bmatrix}\right\},
\end{equation*}
where $a, b \in \mathbb{F}_q$. However, since $u_1, u_2, u_3$ are not equal to $0$, the last one is excluded.
Without loss of generality, we take the first one because $b_{i, j}$ is invariant under a permutation on indices of codewords.
Note that we have $(a, b) \neq (0,0), (1,0), (0,1)$ because $u_1, u_2, u_3 \neq 0$. 
\begin{lem}\label{lem:4-design-rank2}
  If $\rank{A} = 2$, we have the following.
  \begin{enumerate}
    \item If $a + b = 1$ or $ab = 0$, then
          \begin{gather*}
            a_0 = q^{m - 2} - 1,\ 
            a_1 = q^{m - 2}(q - 1),\ 
            a_2 = 3q^{m - 2}(q - 1),\\
            a_3 = q^{m - 2}(q - 1)(4q - 5),\ 
            a_4 = (q - 1)(q^m - 3q^{m - 1} + 2q^{m - 2} - 1).
          \end{gather*}
    \item If $a + b \neq 1$ and $ab \neq 0$, then
          \begin{gather*}
            a_0 = q^{m - 2} - 1,\
            a_1 = 0,\
            a_2 = 6q^{m - 2}(q - 1),\\
            a_3 = q^{m - 2}(q - 1)(4q - 8),\
            a_4 = (q - 1)(q^m - 3q^{m - 1} + 3q^{m - 2} - 1).
          \end{gather*}
  \end{enumerate}
\end{lem}

\begin{proof}
First, we give $b_0$. Clearly, $b_{0, 0} = 0$. For all $j \in\mathbb{F}_q\setminus\{0\}$, 
\begin{align*}
  b_{0, j} &= q^{m - \rank{A}} \times |\{ v \in \Im{A} \mid \text{ zero elements of $v$ are equal to } j\}|\\
           &= q^{m-2} \times |\{ (c_1, c_2) \in \mathbb{F}_q^2 \mid c_1 \neq j,\, c_2 \neq j,\, c_1a+c_2b\neq j\}|.
\end{align*}
We have 
\begin{align*}
   |\{(c_1, c_2)&\in\mathbb{F}_q^2\mid c_1 \neq j, c_2 \neq j, c_1a+c_2b \neq j\}| \\
 = (q - 1)^2 &- |\{(c_1, c_2)\in\mathbb{F}_q^2\mid c_1 \neq j, c_2 \neq j, c_1a+c_2b = j\}|\\
 = (q - 1)^2 &- (
   |\{(c_1, c_2)\in\mathbb{F}_q^2\mid c_1a+c_2b = j\}| \\
 &- |\{(c_1, c_2)\in\mathbb{F}_q^2\mid c_1 = j, c_1a+c_2b = j\}| \\
 &- |\{(c_1, c_2)\in\mathbb{F}_q^2\mid c_2 = j, c_1a+c_2b = j\}| \\
 &+ |\{(c_1, c_2)\in\mathbb{F}_q^2\mid c_1 = j, c_2 = j, c_1a+c_2b = j\}|
 ).
\end{align*}
If $a = 0$, then $b \neq 0, 1$, and if $b = 0$, then $ a \neq 0, 1$. Thus, 
\begin{align*}
  |\{(c_1, c_2)\in\mathbb{F}_q^2\mid c_1a+c_2b = j\}| &= q,\\
  |\{(c_1, c_2)\in\mathbb{F}_q^2\mid c_1 = j, c_1a+c_2b = j\}| &=
  \left\{\begin{array}{ll}
    1 & \text{if } b \neq 0, \\
    0 & \text{if } b = 0,
  \end{array}\right.\\
  |\{(c_1, c_2)\in\mathbb{F}_q^2\mid c_2 = j, c_1a+c_2b = j\}| &=
  \left\{\begin{array}{ll}
    1 & \text{if } a \neq 0, \\
    0 & \text{if } a = 0,
  \end{array}\right.\\
  |\{(c_1, c_2)\in\mathbb{F}_q^2\mid c_1 = j, c_2 = j, c_1a+c_2b = j\}| &= 
  \left\{\begin{array}{ll}
    1 & \text{if } a + b = 1, \\
    0 &  \text{otherwise}.
  \end{array}\right.
\end{align*}
Therefore, 
\begin{align*}
  |\{(c_1, c_2)\in\mathbb{F}_q^2\mid c_1 \neq j, c_2 \neq j, c_1a+c_2b \neq j\}|\\
  = \left\{\begin{array}{ll}
    (q - 1)^2 - (q - 1) & \text{if } ab = 0 \text{ or } a + b = 1 ,\\
    (q - 1)^2 - (q - 2) & \text{otherwise}.
  \end{array}\right.
\end{align*}
Hence,
\begin{align*}
  b_0 &= b_{0, 0} + \sum_{j \in \mathbb{F}_q \setminus\{0\}} b_{0, j} \\
  &= \left\{\begin{array}{ll}
    \sum_{j \in \mathbb{F}_q \setminus\{0\}} q^{m - 2}(q - 1)(q - 2)         & \text{if } a + b = 1 \text{ or } ab = 0 \\
    \sum_{j \in \mathbb{F}_q \setminus\{0\}} q^{m - 2}(q^2 - 3q + 3)  & \text{otherwise}
  \end{array}\right.\\
  &= \left\{\begin{array}{ll}
    q^{m - 2}(q - 1)^2(q - 2)          & \text{if } a + b = 1 \text{ or } ab = 0, \\
    q^{m - 2}(q - 1)(q^2 - 3q + 3)     & \text{otherwise}.
  \end{array}\right.
\end{align*}
Second, we compute $b_1$. Similarly, 
\begin{equation*}
  b_{1, 0} = \left\{\begin{array}{ll}
    q^{m-2}(q - 1)^2           & \text{if } ab = 0,\\
    q^{m-2}(q - 1)(q - 2)      & \text{otherwise}.
  \end{array}\right.
\end{equation*}
Fix $j \in \mathbb{F}_q\setminus\{0\}$, and let $I_1$, $I_2$, and $I_3$ be non-negative numbers such that
\begin{gather*}
  I_1 = |\{ (c_1, c_2) \in \mathbb{F}_q^2 \mid c_1 = j, c_2 \neq j, ac_1 + bc_2 \neq j\}|, \\
  I_2 = |\{ (c_1, c_2) \in \mathbb{F}_q^2 \mid c_1 \neq j, c_2 = j, ac_1 + bc_2 \neq j\}|,\\
  I_3 = |\{ (c_1, c_2) \in \mathbb{F}_q^2 \mid c_1 \neq j, c_2 \neq j, ac_1 + bc_2 = j\}|. 
\end{gather*}
Then, $b_{1, j} = q^{m - 2} \times (I_1 + I_2 + I_3)$. We have
\begin{align*}
  I_1 &= \left\{\begin{array}{lll}
            q - 1 & \text{if } b = 0 \text{ or } a + b = 1,\\
            q - 2 & \text{otherwise},
          \end{array}\right.\\
  I_2 &= \left\{\begin{array}{lll}
            q - 1 & \text{if } a = 0 \text{ or } a + b = 1,\\
            q - 2 & \text{otherwise},
          \end{array}\right.\\
  I_3 &= \left\{\begin{array}{lll}
            q - 1 & \text{if } ab = 0 \text{ or } a + b = 1,\\
            q - 2 & \text{otherwise},
          \end{array}\right.\\
  I_1 + I_2 + I_3 &=   \left\{\begin{array}{ll}
                          3q - 3             & \text{if } a + b = 1,\\
                          3q - 4             & \text{if } ab = 0,\\
                          3q - 6             & \text{otherwise}.
                        \end{array}\right.
\end{align*}
Therefore,
\begin{align*}
  b_1 &= b_{1, 0} + \sum_{j \in \mathbb{F}_q\setminus\{0\}} b_{1, j} \\
      &= \left\{\begin{array}{ll}
        q^{m - 2}(q - 1)(q - 2) + \sum_{j}q^{m - 2}(3q - 3) & \text{if } a + b = 1, \\
        q^{m - 2}(q - 1)^2 + \sum_{j}q^{m - 2}(3q - 4)      & \text{if } ab = 0, \\
        q^{m - 2}(q - 1)(q - 2) + \sum_{j}q^{m - 2}(3q - 6) & \text{otherwise,}
      \end{array}\right.\\
      &= \left\{\begin{array}{ll}
        q^{m - 2}(q - 1)(q - 2) + (q - 1)q^{m - 2}(3q - 3) & \text{if } a + b = 1, \\
        q^{m - 2}(q - 1)^2 + (q - 1)q^{m - 2}(3q - 4)      & \text{if } ab = 0, \\
        q^{m - 2}(q - 1)(q - 2) + (q - 1)q^{m - 2}(3q - 6) & \text{otherwise,}
      \end{array}\right.\\
      &= \left\{\begin{array}{ll}
        q^{m - 2}(q - 1)(4q - 5) & \text{if } a + b = 1 \text{ or } ab = 0, \\
        q^{m - 2}(q - 1)(4q - 8) & \text{otherwise}.
      \end{array}\right.
\end{align*}
Similarly, we give $b_3$. For all $j \in \mathbb{F}_q\setminus\{0\}$, we have
\begin{equation*}
  b_{3, j} = \left\{\begin{array}{ll}
    q^{m - 2} & \text{if } a + b = 1,\\
    0         & \text{otherwise}.
  \end{array}\right.
\end{equation*}
Let $J_1$, $J_2$, and $J_3$ be non-negative integers such that
\begin{gather*}
  J_1 = |\{ (c_1, c_2) \in \mathbb{F}_q^2 \mid c_1 = 0, c_2 \neq 0, ac_1 + bc_2 \neq 0\}|, \\
  J_2 = |\{ (c_1, c_2) \in \mathbb{F}_q^2 \mid c_1 \neq 0, c_2 = 0, ac_1 + bc_2 \neq 0\}|, \\
  J_3 = |\{ (c_1, c_2) \in \mathbb{F}_q^2 \mid c_1 \neq 0, c_2 \neq 0, ac_1 + bc_2 = 0\}|. 
\end{gather*}
Clearly, $J_3 = 0$. We have
\begin{align*}
  J_1 &= \left\{\begin{array}{ll}
    q - 1 & \text{if } a = 0 ,\\
    0     & \text{otherwise},
  \end{array}\right.\\
  J_2 &= \left\{\begin{array}{ll}
    q - 1 & \text{if } b = 0 ,\\
    0     & \text{otherwise}.
  \end{array}\right.
\end{align*}
Thus,
\begin{align*}
  b_{3, 0} = q^{m - 2}\times(J_1 + J_2 + J_3)
  = \left\{\begin{array}{ll}
    q - 1 & \text{if } ab = 0, \\
    0     & \text{otherwise}.
  \end{array}\right.
\end{align*}
Therefore,
\begin{align*}
  b_3 &= b_{3, 0} + \sum_{j \in \mathbb{F}_q \setminus \{0\}} b_{3, j} \\
      &= \left\{\begin{array}{ll}
        q^{m - 2} (q - 1) & \text{if } a + b = 1 \text{ or } ab = 0,\\
        0                 & \text{otherwise}        ,
      \end{array}\right.
\end{align*}
and we have $b_{4, j} = 0$, where $j \in \mathbb{F}_q\setminus\{0\}$, and $b_{4, 0} = |\Ker{A}| = q^{m-2}$.
Then we have
\begin{equation*}
  b_4 = q^{m-2}.
\end{equation*}
Finally, by using Lemma~\ref{lem:b-to-a}, we obtain $a_0, \dots, a_4$ as follows:
\begin{align*}
  a_0 &= b_4 - 1 \\
      &= q^{m - 2} - 1, \\
  a_1 &= b_3 \\
      &= \left\{\begin{array}{ll}
        q^{m - 2} (q - 1) & \text{if } a + b = 1 \text{ or } ab = 0,\\
        0                 & \text{otherwise},
      \end{array}\right. \\
  a_3 &= b_1 \\
      &= \left\{\begin{array}{ll}
        q^{m - 2}(q - 1)(4q - 5) & \text{if } a + b = 1 \text{ or } ab = 0, \\
        q^{m - 2}(q - 1)(4q - 8) & \text{otherwise},
      \end{array}\right.\\
  a_4 &= b_0 - (q - 1) \\
      &= \left\{\begin{array}{ll}
        q^{m - 2}(q - 1)^2(q - 2) - (q - 1)          & \text{if } a + b = 1 \text{ or } ab = 0, \\
        q^{m - 2}(q - 1)(q^2 - 3q + 3) - (q - 1)     & \text{otherwise},
      \end{array}\right.\\
      &= \left\{\begin{array}{ll}
        (q - 1)(q^m - 3q^{m - 1} + 2q^{m - 2} - 1)          & \text{if } a + b = 1 \text{ or } ab = 0, \\
        (q - 1)(q^m - 3q^{m - 1} + 3q^{m - 2} - 1)          & \text{otherwise},
      \end{array}\right.\\
  a_2 &= |C \setminus \mathbb{F}_q\mathbf{1}| - (a_0 + a_1 + a_3 + a_4) \\
      &= q^{m + 1} - q - (a_0 + a_1 + a_3 + a_4)\\
      &= \left\{\begin{array}{ll}
        3q^{m - 2}(q - 1) & \text{if } a + b = 1 \text{ or } ab = 0, \\
        6q^{m - 2}(q - 1) & \text{otherwise}.
      \end{array}\right.
\end{align*}

\end{proof}

\begin{proof}[Proof of Theorem~\ref{thm:4-design} and Corollary~\ref{cor:4-gen-design}]
  Using Lemmas~\ref{lem:4-design-rank3}, \ref{lem:4-design-rank1}, and \ref{lem:4-design-rank2}, we obtain 
  the desired results.
\end{proof}

\section*{Acknowledgments}
The authors would like to thank the anonymous reviewers for their beneficial comments
on an earlier version of the manuscript.
We are grateful to Professors Tsuyoshi Miezaki and Akihiro Munemasa
for helpful discussions and comments on the manuscript.

\section*{Statements and declarations}
Competing Interests. The authors have no affiliation with any organization with a 
direct or indirect financial interest in the subject matter discussed in the manuscript.

\end{document}